\documentclass[preprint,12pt,fleqn]{elsarticle}

\usepackage{amssymb,amsfonts,amsmath,amscd,latexsym,amsthm}
\usepackage[english]{babel}
\usepackage{graphicx}
\usepackage{mathptmx}      
\usepackage{geometry}
\usepackage[all]{xy}
\usepackage{appendix}
\usepackage[active]{srcltx}
\usepackage{color}

\newtheorem{thm}{Theorem}[section]

\newtheorem{lem}[thm]{Lemma}
\newtheorem{prop}[thm]{Proposition}
\theoremstyle{definition}
\newtheorem{defn}[thm]{Definition}
\newtheorem{subsec}[thm]{}
\newtheorem{nim}[thm]{}
\newtheorem{rem}[thm]{Remark}

\newcommand{\Tr}{\operatorname{Tr}}
\newcommand{\End}{\operatorname{End}}

\newcommand{\Aut}{\operatorname{Aut}}

\newcommand{\Br}{\operatorname{Br}}

\newcommand{\Sum}{\displaystyle\sum}

\begin{document}

\begin{frontmatter}
\title{Remarks on the extended Brauer quotient}
\author{Tiberiu Cocone\c t} \ead{tiberiu.coconet@math.ubbcluj.ro}
\address{Faculty of Economics and Business Administration, Babe\c s-Bolyai University, Str. Teodor Mihali, nr.58-60 , 400591 Cluj-Napoca, Romania}
\author{Andrei Marcus} \ead{andrei.marcus@math.ubbcluj.ro}
\address{Faculty of Mathematics, Babe\c s-Bolyai University, Cluj-Napoca, Romania}
\cortext[cor1]{Corresponding author}

\begin{abstract}

\end{abstract}

\end{frontmatter}

\section{Introduction}\label{sec1}

Let $G$ be a group, and let $A$ be a $G$-algebra over a complete discrete valuation ring $\mathcal{O}$ with residue field $k$ of characteristic $p>0$. The well-known Brauer quotient $A(P)$ with respect to a $p$-subgroup $P$ of $G$ (introduced by M.~Brou\'e and L.~Puig, see \cite[\S 11]{The}) is an $N_G(P)$-algebra. If moreover, $A$ is $G$-interior (that is, $A$ is endowed with a unitary algebra homomorphism $\mathcal{O}G\to A$), then $A(P)$ becomes a $C_G(P)$-interior $N_G(P)$-algebra. This means that one may construct, as in \cite[Chapter 9]{Pu}, the $N_G(P)/C_G(P)$-graded $N_G(P)$-interior algebra $A(P)\otimes_{C_G(P)}N_G(P)$, so $A(P)$ is extended by automorphisms of $P$ given by conjugation with elements of $G$.

L.~Puig and Y.~Zhou \cite{PuZh} extended $A(P)$ by all automorphisms of $P$, obtaining the so called {\it extended Brauer quotient} $\bar N_A^{\Aut(P)}(P)$ as an $N_G(P)$-interior $k$-algebra. The interiority assumption is necessary, because the main feature used is the $\mathcal{O}(P\times P)$-module structure of $A$. This construction was further generalized by T.~Cocone\c t and C.-C.~Todea \cite{COTO} to the case of $H$-interior $G$-algebras, where $H$ is a normal subgroup of $G$.

Our aim here is to unify and generalize these constructions, by introducing an extended Brauer quotient of a group graded algebra. The main ingredients of our construction are a $\bar G$-graded algebra $A$, a group homomorphism $P\to \bar G$ (which induces a $\bar G$-grading on the group algebra $\mathcal{O}P$), and a homomorphism $\mathcal{O}P\to A$ of $\bar G$-graded algebras.

In Section 2 below we recall the Puig and Zhou definition of $\bar N_A^{\Aut(P)}(P)$, pointing out its $\Aut(P)$-graded algebra structure. Our alternative construction in Section 3 is based on the easy observation that if $A$ is a $\bar G$-graded $P$-algebra with identity component $B$ such that the action of $P$ on $\bar G$ is trivial, then the Brauer quotient $A(P)$ inherits the $\bar G$-grading such that the identity component of $A(P)$ is $B(P)$. Here we apply the classical Brauer quotient to the $\Aut(P)$-graded algebra $\tilde A=A\otimes_{\mathcal{O}}\mathcal{O}\Aut(P)$, and we get that $\tilde A(P)$ is isomorphic to $\bar N_A^{\Aut(P)}(P)$ as $\Aut(P)$-graded algebras. In Section 4 we construct the extended Brauer quotient of a $\bar G$-graded $P$-interior algebra $A$ as mentioned above, this time with $P$ acting nontrivially on $\bar G$. We also discuss the exact relationship to the construction from \cite{COTO}. Section 5 investigates the extended Brauer quotient of tensor products of $P$-interior algebras,  in Section 6 we give an application towards correspondences for covering blocks.

Our general notations and assumptions are standard, and  closely follow \cite{The}, \cite{Pu} and \cite{Ma}.

\section{The extended Brauer quotient} \label{sec2}
\subsection{The  construction of Puig and Zhou}\label{general_const}

We begin with a $p$-group $P$ and a $P$-interior algebra $A.$ Let $\varphi\in \Aut(P)$, and as in \cite{PuZh},  we consider the $\varphi$-twisted  diagonal
\[\Delta_{\varphi}(P)=\{(u,\varphi(u))\mid u\in P\}.\] Then the set of  $\Delta_{\varphi}(P)$-fixed elements, is the following $\mathcal{O}$-submodule of $A$:
 \[A^{\Delta_{\varphi}(P)}=\{a\in A\mid ua=a\varphi(u) \mbox{ for any } u\in P\}.\]
Further, we consider $Q<P$ and denote by $A_{\Delta_{\varphi}(Q)}^{\Delta_{\varphi}(P)}$ the $\mathcal{O}$-module consisting of elements of the form
   \[\Tr_{\Delta_{\varphi}(Q)}^{\Delta_{\varphi}(P)}(c)=\sum_{u\in [P/Q]}u^{-1}c\varphi(u),\] where $c\in A^{\Delta_{\varphi}(Q)}.$
At last, we denote by $A(\Delta_{\varphi}(P))$ the quotient
\[A(\Delta_{\varphi}(P))=A^{\Delta_{\varphi}(P)}/\Sum_{Q<P}A_{\Delta_{\varphi}(Q)}^{\Delta_{\varphi}(P)},\]
and we obtain the usual Brauer homomorphism
\[\Br_{\Delta_{\varphi}(P)}:A^{\Delta_{\varphi}(P)}\to A(\Delta_{\varphi}(P)).\]
If $K$ is a subgroup of $\Aut(P)$, it is easily checked that the external direct sum $\bigoplus_{\varphi\in K}A^{\Delta_{\varphi}(P)}$ is an algebra, while its subset
$\bigoplus_{\varphi\in K}\sum_{Q<P}A^{\Delta_{\varphi}(P)}_{\Delta_{\varphi}(Q)}$
is a two-sided ideal, hence we have the following definition.

\begin{defn}[{\cite{PuZh}}]  The {\it extended Brauer quotient} associated to the $P$-interior algebra $A$ and the subgroup $K$ of $\Aut(P)$ is the external direct sum
\[\bar{N}^K_A(P):=\bigoplus_{\varphi\in K}A^{\Delta_{\varphi}(P)}/\bigoplus_{\varphi\in K}\sum_{Q<P}A^{\Delta_{\varphi}(P)}_{\Delta_{\varphi}(Q)}\simeq \bigoplus_{\varphi\in K} A(\Delta_{\varphi}(P)).\]
\end{defn}

\begin{rem} Note that in this case, one  deduces easily from the details given in \cite[Section 3]{PuZh} and \cite[Section 3]{PuZhII} that $\bar{N}^K_A(P)$ is a $K$-graded algebra, and the map $\Br_P^K:=\bigoplus_{\varphi\in K}\Br_{\Delta_{\varphi}(P)}$ is a homomorphism of $K$-graded algebras. This fact will become even more transparent in the next section.
\end{rem}

\subsection{The case of  $G$-interior algebras}

In addition to the situation of subsection \ref{general_const} we assume the $A$ is a $G$-interior algebra, where $G$ is a (not necessarily finite) group, and $P$ is a $p$-subgroup of $G.$  Conjugation induces the  the group homomorphisms
\[N_G(P)\to \Aut(P) \qquad\mathrm{and}\qquad N_G(P)/C_G(P)\to \Aut(P),\tag{1}\label{group_homo}\] and for the subgroup $K$ in $\Aut(P),$ $N_G^K(P)$ denotes the inverse image of $K$ in $N_G(P)$. If $x\in N_G(P)$, we use denote by $\varphi_x$ the automorphism of $P$ given by $\varphi_x(u)=u^x=x^{-1}ux$ for all $u\in P.$

In this setting, we obtain some additional properties of the extended Brauer quotient (the details are   left  to the reader).

\begin{prop} \label{p:ngpcgp}  With the above notation, the following statements hold:
\begin{enumerate}
\item[{\rm1)}] $\bar{N}^K_A(P)$ is a $K$-graded $N^K_G(P)$-interior algebra;
\item[{\rm2)}] If $K=N_G(P)/C_G(P)$, then we have the isomorphism
          \[\bar{N}^K_{A}(P)\simeq A(P)\otimes_{kC_G(P)}kN_G(P)\] of $N_G(P)/C_G(P)$-graded $N_G(P)$-interior algebras.
\end{enumerate}\end{prop}

\begin{proof} 1)  We only need to notice that any $x\in N_G^K(P)$ verifies $u^{-1}x\varphi_x(u)=x.$


2) We  define the $N_G(P)/C_G(P)$-graded map
\[A(P)\otimes_{kC_G(P)}kN_G(P) \to \bar{N}^K_{A}(P), \qquad \bar{a}\otimes x\mapsto \overline{ax},\]
whose restriction to the identity component is an isomorphism.
\end{proof}

\begin{rem} Note that if $K=N_G(P)/C_G(P)$, then $\bar{N}^K_{\mathcal{O}G}(P)$ is just the group algebra $kN_G(P)$ considered with the obvious $K$-grading. Moreover, the construction of $\bar{N}^K_{A}(P)$ is clearly functorial in $A$, so the $N_G(P)$-interior algebra structure of $\bar{N}^K_{A}(P)$ comes from applying the construction to the algebra map $\mathcal{O}G\to A$.
\end{rem}

\section{An alternative construction}

\begin{nim} The $\mathcal{O}P$-interior algebra $A$ admits an obvious $(\mathcal{O}P,\mathcal{O}P)$-bimodule structure. Consider the group algebra $\mathcal{O}[P\rtimes \Aut(P)],$ of the semidirect product $P\rtimes \Aut(P).$ This algebra is also a left $\mathcal{O}P$-module, hence it makes sense to consider the $\Aut(P)$-graded $(\mathcal{O}P,\mathcal{O}P)$-bimodule
\[\tilde{A}:=A\otimes_{P}\mathcal{O}(P\rtimes \Aut(P)).\]
We may also use the isomorphism
\[\tilde A\simeq A\otimes_{\mathcal{O}}\mathcal{O} \Aut(P)\]
of $\mathcal{O}$-modules, which becomes an isomorphism of $(\mathcal{O} P, \mathcal{O} P)$-bimodules, by defining the bimodule structure of $A\otimes_{\mathcal{O}}\mathcal{O} \Aut(P)$ as follows:
\[u(a\otimes \phi)v=u\cdot a\cdot\phi(v)\otimes \phi,\]
for $u,v\in P$ and $\phi\in \Aut(P)$. Then we regard $A\otimes_{\mathcal{O}}\mathcal{O} \Aut(P)$ as an $\Aut(P)$-graded $P$-algebra with $P$-action given by
\[(a\otimes \phi)^u=u^{-1}\cdot a\cdot\phi(u)\otimes \phi,\]
\end{nim}

With the notations of Sections \ref{sec2} and 3 we have:

\begin{thm} \label{theisomorphismofquotients} There  is an isomorphism
\[\tilde{A}(P)\simeq \bar{N}_A^{\Aut(P)}(P)\]
of $\Aut(P)$-graded algebras, where $\tilde{A}(P)$ is the usual Brauer quotient of $\tilde{A}$.
\end{thm}

\begin{proof} As the $p$-group  $P$ is a  normal subgroup of   $P\rtimes \Aut(P)$, we get the decomposition
\[\tilde{A}(P)=\bigoplus_{\varphi\in \Aut(P)}(A\otimes (1,\varphi))(P).\]
If $a\otimes (1,\varphi)\in (A\otimes (1,\varphi))^{P}$, then
\[u^{-1}\cdot a\otimes (1,\varphi)\cdot u=u^{-1}\cdot a\otimes (1,\varphi) (u,1)=u^{-1}a\varphi(u)\otimes (1,\varphi)=a\otimes (1,\varphi).\]  Then $a\in A^{\Delta_{\varphi}(P)},$ and consequently
\[(A\otimes (1,\varphi))(P) \to \bar{N}_A^{\varphi}(P), \qquad \overline{a\otimes (1,\varphi)}\mapsto \overline{a},\]
is a well-defined map of $\mathcal{O}$-modules for every $\varphi\in \Aut(P).$ We extend this map to a $\Aut(P)$-graded map between these two modules and we notice that, with all the above identifications, it is actually an isomorphism of algebras.
\end{proof}

\begin{rem} We often use subgroups of $P$, and we obviously have the isomorphism
\[(A\otimes_{\mathcal{O}Q} \mathcal{O}[Q\rtimes K])(Q)\simeq \bar{N}_A^K(Q)\]
of $K$-graded algebras, for any subgroups $Q \leq P$ and $K\leq \Aut(Q).$ \end{rem}

\section{The extended Brauer quotient of a group graded algebra} \label{s:graded}

In this paragraph we set $\bar{G}:=G/H,$ where $H$ is a normal subgroup of the finite group $G,$ $P$ is a $p$-subgroup of $G$, and let
\[A:=B\otimes_{\mathcal{O}H}\mathcal{O}G\] for some $H$-interior $G$-algebra $B$, so $A$ is the $G$-interior $\bar G$-graded algebra induced from $B$.

The following lemma says that  we restrict ourselves, without loss, to a certain   subgroup  of $\Aut(P)$.

\begin{lem}\label{zero_quotient_prop} Let $\varphi\in \Aut(P)$, and let $O(\bar{x})$ be the orbit of $\bar{x}\in \bar G$ under the action of $\Delta_{\varphi}(P)$ on $\bar{G}.$  If  $\mid O(\bar{x}) \mid \neq 1$ then $(\bigoplus_{\bar{z}\in O(\bar{x})}B\otimes z)(\Delta_{\varphi}(P))= 0.$
\end{lem}

\begin{proof} Consider the element $a=\Sum b_{z_i}\otimes z_i$ such that $u^{-1}a\varphi (u)=a.$ Since the elements $z_i$ are all representatives of the classes of an orbit, we can choose them such that for any $u\in P$ we obtain $u^{-1}z_i\varphi(u)=z_j.$ It follows that $b_{z_i}^u=b_{z_j}$, and then  there is one element, say $b_z,$ such that $b_z^u=b_{z_i}$ for any $i$ and any $u\in P.$ Hence \[a=\Tr_{\Delta_{\varphi}(Q)}^{\Delta_{\varphi}(P)}(b_z\otimes z),\] where $Q$ is the stabilizer of $b_z\otimes z$ in $P.$
\end{proof}

\begin{subsec} {\rm The above lemma gives the motivation to introduce two subgroups of $\Aut(P)$, because it implies that $(A\otimes\phi)(P)=0$ for $\phi\in \Aut(P)$ not satisfying $\overline{\phi(u)}=\overline{u^g}$ in $\bar G$, for some $g\in G$. So let
\[\Aut_{\bar G}(P)=\{\varphi\in \Aut(P)\mid \overline{\varphi(u)}=\overline{u^{{g}}} \mbox{ for some } \bar{g}\in \bar{G} \mbox{ and for any } u\in P\}\]
and
\[\Aut_{\bar 1}(P)=\{\varphi\in \Aut(P)\mid \overline{\varphi(u)}= \bar u \mbox{ for any } u\in P\}.\]
Denote also \[K:=\Aut_{\bar G}(P), \qquad K_1:=\Aut_{\bar 1}(P),\] and let $\tilde{A}:=A\otimes_{\mathcal{O}P}\mathcal{O}[P\rtimes K]$ as in Section 3.

Finally, let $N^K_{\bar{G}}(\bar{P})$ denote the subgroup of $N_{\bar{G}}(\bar{P})$ whose elements define an element of $K$ and let $U$ be the inverse image of $N^K_{\bar{G}}(\bar{P})$ in $G.$ Also let $U'$ be the inverse image of $C_{\bar{G}}(\bar{P})$ in $G$. Observe that $N^K_G(P)=N_G(P)$ and $N^{K_1}_G(P)=N_G(P)\cap U'$.
}\end{subsec}

\begin{lem} The group $\Aut_{\bar 1}(P)$ is a normal subgroup of $\Aut_{\bar G}(P)$, hence $U'$ is normal in $U.$ Furthermore, we have the isomorphisms
\[\Aut_{\bar G}(P)/\Aut_{\bar 1}(P)\simeq N^K_{\bar{G}}(\bar{P})/C_{\bar{G}}(\bar{P})\simeq U/U'.\]
\end{lem}

\begin{proof}  If $\varphi_1\in \Aut_{\bar 1}(P)$ then $\overline{\varphi(u)}=\bar{u}$ for all $u\in P.$ Hence, if $\varphi\in \Aut_{\bar G}(P)$ with $\overline{\varphi(u)}=\overline{u^{\bar{g}}},$ we get
\[\overline{(\varphi^{-1}\circ \varphi_1 \circ \varphi)(u)}=\overline{(\varphi_1 \circ \varphi)(u)^{g^{-1}}}=\overline{\varphi(u)^{g^{-1}}}=\bar{u}.\]
Further if $x\in \bar{g}$ then $\overline{\varphi(u)}=\overline{u^{g}}=\overline{u^{x}}$ and then $\overline{g^{-1}x}\in C_{\bar{G}}(\bar{P}).$ With all of the above, the map
\[\Aut_{\bar G}(P)\ni \varphi \mapsto \bar{g}\in  N_{\bar{G}}(\bar{P}) \]
gives the first isomorphism. The second isomorphism is obvious.
\end{proof}

We will denote by $\tilde{\varphi}$ the image of $\varphi$ in the quotient group $\Aut_{\bar G}(P)/\Aut_{\bar 1}(P)$.
\begin{thm}\label{graded_quotient_version} The algebra $\bar{N}^{\Aut(P)}_A(P)$, as constructed in \ref{general_const}, is the $U/U'$-graded $N_G(P)$-interior algebra $\bar{N}^{K}_A(P)$ with identity component the $N^{K_1}_G(P)$-interior algebra
\[\bar{N}^{K_1}_A(P)=\bigoplus_{ g'\in U'/H}\bar{N}^{K_1}_{B\otimes g'}(P),\]
and for any $\tilde{g}\in U/U'$ (corresponding to $\tilde{\varphi}$), the $\tilde{g}$-component is
\[\bar{N}^{K}_A(P)_{\tilde{g}}=\bigoplus_{\varphi_{\bar{z}}\in \tilde{\varphi};\ {\bar z}\in \tilde{g}}(B\otimes z)(\Delta_{\varphi_{\bar{z}}}(P)),\]
where  $\varphi_{\bar z}\in \tilde \varphi$ satisfies $\overline{\varphi_{\bar{z}}(u)}=\overline{u^z},$ for any $u\in P.$
\end{thm}

\begin{proof}  By Lemma \ref{zero_quotient_prop}, we obtain the following decomposition of the extended Brauer quotient
\begin{align*}\bar{N}^{\Aut(P)}_A(P)&=\left(\bigoplus_{\varphi\in K_1}\bar{N}^{\varphi}_A(P)\right) \oplus \bar{N}^{K\setminus K_1}_{A}(P)\\
                &= \left(\bigoplus_{\bar g'\in U'/H}\bar{N}^{K_1}_{B\otimes g'}(P)\right)     \\
                &\oplus     \bigoplus_{\tilde{\varphi}\in K/K_1 }\left(\bigoplus_{\varphi_{\bar{z}}\in \tilde{\varphi};\ \bar{z}\in \tilde{g}}(B\otimes z )(\Delta_{\varphi_{\bar{z}}}(P))\right),
\end{align*}
where in the second sum $\tilde{\varphi}$ corresponds to $\tilde{g}.$

We see that, for any $\tilde{g}$ and any $\bar{z}\in \tilde{g},$
\[B\otimes z=B\otimes zx^{-1}x=(B\otimes zx^{-1})\cdot (B\otimes x),\]
for any $x\in U'.$ Then
\[\bar{N}^{K}_A(P)_{\tilde{g}}\cdot \bar{N}^{K_1}_A(P)= \bar{N}^{K_1}_A(P)\cdot \bar{N}^{K}_A(P)_{\tilde{g}}=\bar{N}^{K}_A(P)_{\tilde{g}}. \]
The fact that this algebra is $N_G(P)$-interior is immediate since for any $x\in N_G(P)$ the element $1\otimes x$ is $\Delta_{\varphi_{\bar{x}}}(P)$-invariant.
\end{proof}

\begin{rem} \label{r:tildeK} 1) The fact that in the above theorem every $\tilde g$-component of $\bar N_A^K(P)$ is a direct sum suggests that this algebra actually has a finer grading than stated. Indeed, it is not difficult to see that $\bar N_A^{K_1}(P)$ is graded by the group
\[\tilde K_1:=\{(\phi, \bar g) \mid \phi\in K_1, \  \bar g\in U'/H \textrm{ such that } \overline{\phi(u)}=\overline{u^g}\},\]
and in general, $\bar N_A^K(P)$ is graded by the group
\[\tilde K:=\{(\phi, \bar g) \mid \phi\in K, \  \bar g\in U/H \textrm{ such that } \overline{\phi(u)}=\overline{u^g}\}.\]

2) Applying the construction to the group algebra $\mathcal{O}G$ yields $\bar N_{\mathcal{O}G}^K(P)=kN_G(P)$. The map
\[N_G(P)\to \tilde K, \qquad g\mapsto (\phi_g,\bar g),\]
where $\phi_g(u)=gug^{-1}$, is a group homomorphism with kernel $C_H(P)$. The $\bar G$-graded algebra map $\mathcal{O}G \to A$ induces by functoriality the $\tilde K$-graded algebra map \[kN_G(P)\to \bar N_A^K(P).\]

3) Observe finally that the construction  of $\bar N_A^K(P)$ does not require the $G$-interiority of $A$. We only need a $\bar G$-graded algebra $A$, a group homomorphism $P\to \bar G$  inducing a $\bar G$-grading on the group algebra $\mathcal{O}P$, and a homomorphism $\mathcal{O}P\to A$ of $\bar G$-graded algebras.
\end{rem}


\begin{nim} Next, we establish the connection between $\bar N_A^K(P)\simeq\tilde A(P)$ and the extended Brauer quotient $\bar{N}^{K_1}_B(P)$ of the $H$-interior $G$-algebra $B$, introduced in \cite{COTO}. Recall that $\bar{N}^{K_1}_B(P)$ is an $N_H^{K_1}(P)$-interior $N_G(P)$-algebra constructed formally as in Section 1 above. One can easily see from the definition in \cite[Section 2]{COTO} that $\bar{N}^{K_1}_B(P)$ is actually a $K_1$-graded $N_G^{K_1}(P)$-interior $N_G(P)$-algebra.

Let $Q:=P\cap H$.   Then, as in Section 3,  let \[\tilde{B}:=B\otimes_{\mathcal{O}Q}\mathcal{O}(Q\rtimes K_1)\simeq B\otimes_{\mathcal{O}}\mathcal{O}K_1.\]
\end{nim}

\begin{prop} The $\mathcal{O}$-module $\tilde{B}$ is a $\mathcal{O}\Delta(P\times P)$-module via
\[(b\otimes (1,\varphi))^{(u,u)}=b^u\otimes u^{-1}\cdot (1,\varphi) \cdot u =b^u\otimes (u^{-1}\varphi(u),\varphi)=b^uu^{-1}\varphi(u)\otimes (1,\varphi),\]
for any $u\in P,$ $b\in B$ and $\varphi\in K_1.$ Furthermore, we have the  isomorphism
\[\tilde{B}(\Delta(P\times P))\simeq \bar{N}^{K_1}_B(P)\]
of $K_1$-graded  $N^{K_1}_G(P)$-interior $N_G(P)$-algebras with identity component $B(P)$.
\end{prop}

\begin{proof}
It is clear that for any $\varphi\in K_1$ we have $\varphi(u)\in Q$ for any $u\in Q,$ hence  $K_1$ acts on $Q$ and $\tilde{B}$ is a well-defined $\Delta(P\times P)$-module and we have
\[\tilde{B}(\Delta(P\times P))=\bigoplus_{\varphi\in K_1}(B\otimes (1,\varphi))(\Delta(P\times P)).\]
For any $\varphi\in K_1$ the map
\[(B\otimes (1,\varphi))(\Delta(P\times P))\ni \overline{b\otimes (1,\varphi)}\mapsto \bar{b}\in \bar{N}^{\varphi}_B(P)\] is an isomorphism of $k$-vector spaces. The direct sum of these maps is the required algebra isomorphism.
\end{proof}

\begin{rem} 1) According to Theorem \ref{theisomorphismofquotients} and  Theorem \ref{graded_quotient_version},  we have the decompositions
\begin{align*}
\tilde{A}(P)\simeq \bar{N}^K_A(P)&=\bar{N}^{K_1}_{B\otimes_H U'}(P)\oplus \bar{N}^{K\setminus K_1}_{A}(P)   \\
                                 &=\bar{N}^{K_1}_B(P)\oplus\left( \bigoplus_{\bar x\in U'/H,\ x\notin H} \bar{N}^K_{\substack{B\otimes x}}(P) \right) \oplus \bar{N}^{K\setminus K_1}_{A}(P).
\end{align*}
The above statements show that the $N_G^{K_1}(P)$-interior algebra $\tilde{B}(P)$ can be identified with a unitary subalgebra of $\tilde{A}(P)$, and even of $\bar N_A^{K_1}(P)$, such that the $N_G(P)$-action and the $K_1$-grading are preserved. For the particular case of the $H$-interior $G$-invariant group algebra $B=\mathcal{O}H$, the component $\bar{N}^{K_1}_{B\otimes_H U'}(P)$ is the $N_G(P)$-algebra studied in \cite[Section 5]{CoMaBasicGraded}.

2) The Brauer quotient $B(P)$ of $B$ is a $C_H(P)$-interior $N_G(P)$-algebra. The argument of Proposition \ref{p:ngpcgp} implies that the induced algebra \[B(P)\otimes_{kC_H(P)}kN_G(P)\] is isomorphic to a $\tilde K$-graded subalgebra of $\bar N_A^K(P)$, while  \[B(P)\otimes_{kC_H(P)}kC_G(P)\] is isomorphic to a $\tilde K_1$-graded subalgebra of $\bar N_B^K(P)$.
\end{rem}

\section{Tensor products of algebras}

Recall that if $A$ and $A'$ are two $G$-graded algebras, then the diagonal subalgebra of the $G\times G$-graded algebra $A\otimes A'$ is the $G$-graded subalgebra
\[\Delta(A\otimes A'):=\bigoplus_{g\in G}(A_g\otimes A'_g) .\]
The following result is an extension of \cite[Proposition 3.9]{PuZh}

\begin{thm}\label{prop_ext_3.9} Assume that $A$ and $A'$ are two $G$-interior algebras such that $A'$ has a $P\times P$-invariant $\mathcal{O}$-basis,  and let $K$ be a subgroup of $\Aut(P).$

{\rm 1)} There is an isomorphism
\[\bar{N}^K_{A\otimes_{\mathcal{O}}A'}(P)\simeq \Delta(\bar{N}_A^K(P)\otimes_k \bar{N}_{A'}^K(P))\]
of $K$-graded $N_G^K(P)$-interior algebras.

{\rm 2)} Assume in addition that, as $K$-graded $N^K_G(P)$-interior algebras,
\[\bar{N}_A^K(P)\simeq A(P)\otimes_kkK.\] Then
\[\bar{N}^K_{A\otimes_{\mathcal{O}}A'}(P)\simeq A(P)\otimes_k \bar{N}_{A'}^K(P)\] as $K$-graded $N_G(P)$-interior algebras.
\end{thm}

\begin{proof} 1) We consider the  $K\times K$-graded $N^K_G(P)$-interior algebra
\[\bar{N}_A^K(P)\otimes_k \bar{N}_{A'}^K(P)=\bigoplus_{\varphi,\psi\in K}\bar{N}_A^{\varphi}(P)\otimes_k \bar{N}_{A'}^{\psi}(P),\] whose diagonal subalgebra
\[\Delta(\bar{N}_A^K(P)\otimes_k \bar{N}_{A'}^K(P))=\bigoplus_{\varphi\in K}\bar{N}_A^{\varphi}(P)\otimes_k \bar{N}_{A'}^{\varphi}(P)\] is an  $N^K_G(P)$-interior  $K$-graded algebra.
Due to the inclusion
\[A^{\Delta_{\varphi}(P)}\otimes_{\mathcal{O}} (A')^{\Delta_{\varphi}(P)}\subseteq (A\otimes_{\mathcal{O}} A')^{\Delta_{\varphi}(P)},\]
we obtain an $\mathcal{O}$-module map
\[A^{\Delta_{\varphi}(P)}\otimes_{\mathcal{O}} (A')^{\Delta_{\varphi}(P)}\to \bar{N}^{\varphi}_{A\otimes_{\mathcal{O}}A'}(P)\]
 sending $a\otimes a'$ to $\overline{a\otimes a'}.$ If $c\in A^{\Delta_{\varphi}(Q)}$ and $c'\in (A')^{\Delta_{\varphi}(R)},$ for some subgroups $Q$ and $R$ of $P$ then
 \begin{align*} 
\Tr_{\Delta_{\varphi}(Q)}^{\Delta_{\varphi}(P)}(c)\otimes \Tr_{\Delta_{\varphi}(R)}^{\Delta_{\varphi}(P)}(c')& =\Tr_{\Delta_{\varphi}(Q)}^{\Delta_{\varphi}(P)}\left(c\otimes \Tr_{\Delta_{\varphi}(R)}^{\Delta_{\varphi}(P)}(c')\right)  \\
     &= \Tr_{\Delta_{\varphi}(R)}^{\Delta_{\varphi}(P)}\left(\Tr_{\Delta_{\varphi}(Q)}^{\Delta_{\varphi}(P)}(c)\otimes c'\right)   \\ 
     &\in (A\otimes_{\mathcal{O}}A')_{\Delta_{\varphi}( R)}^{\Delta_{\varphi}(P)}\cap (A\otimes_{\mathcal{O}}A')_{\Delta_{\varphi}( Q)}^{\Delta_{\varphi}(P)} .\end{align*} 
This determines an $\mathcal{O}$-module homomorphism
\[\bar{N}_A^{\varphi}(P)\otimes \bar{N}_{A'}^{\varphi}(P) \to \bar{N}^{\varphi}_{A\otimes_{\mathcal{O}}A'}(P), \qquad  \bar{a}\otimes \bar{a'}\mapsto \overline{a\otimes a'}\] 
for every $\varphi\in K.$ The direct sum of all these homomorphism is a $K$-graded algebra homomorphism between $\Delta(\bar{N}_A^K(P)\otimes_k \bar{N}_{A'}^K(P))$ and $\bar{N}^K_{A\otimes_{\mathcal{O}}A'}(P),$  which is in fact an isomorphism of interior $N^K_G(P)$-algebras since by our assumptions we have \[(A\otimes_{\mathcal{O}}A')(P)\simeq A(P)\otimes_k A'(P).\]

2) By the additional assumption we obtain
\[\Delta(\bar{N}_A^K(P)\otimes_k \bar{N}_{A'}^K(P))=\bigoplus_{\varphi\in K}(A(P)\otimes_k k\varphi)\otimes \bar{N}^{\varphi}_{A'}(P).\]
We define the $k$-linear map
\[A(P)\otimes_k \bar{N}_{A'}^{\varphi}(P)\to (A(P)\otimes_k k\varphi)\otimes_k\bar{N}_{A'}^{\varphi}(P), \qquad \bar{a}\otimes \bar{a'}\mapsto (\bar{a}\otimes \varphi)\otimes \bar{a'},\]
for  every $\varphi\in K.$ The sum of these maps determine the required isomorphism of $K$-graded interior $N_G(P)$-algebras between $A(P)\otimes_k \bar{N}_{A'}^{K}(P)$ and $\Delta(\bar{N}_A^K(P)\otimes_k \bar{N}_{A'}^K(P)).$
\end{proof}

\begin{rem} 1) Statement 2) of the previous theorem applies in the  situation of \cite[Proposition 3.8]{PuZh}. More precisely, let
\[A=\End_{\mathcal{O}}(N)\] for an indecomposable endopermutation $\mathcal{O}P$-module $N$, such that $A(P)\neq 0$. Let $Q\leq P$, and let $\delta$ be the unique local point of $Q$ on $A$. Let $K:=F_A(Q_\delta)$.
Then \cite[Proposition 3.8]{PuZh} says that there is an isomorphism
\[\bar N_A^K(Q) \simeq A(Q)\otimes_kkK  \]
of $N^K_P(Q)$-interior $K$-graded algebras.

2) Assume in addition that $A'$ is $\bar G$-graded $G$-interior as in Section 4, and has a $P\times P$-invariant $\mathcal{O}$-basis consisting of $\bar G$-homogeneous elements. Then, by Remark \ref{r:tildeK}, the isomorphism in Theorem \ref{prop_ext_3.9}. 2) is in fact an isomorphism of $\tilde K$-graded $N_G(P)$-interior algebras.
\end{rem}



\section{A correspondence for  covering points}

In this section we establish a correspondence between covering points in the case of a $G$-interior algebra that has a stable basis.

\begin{nim} We keep the notations of Section \ref{s:graded}, and we assume that the $G$-interior $\bar G$-graded algebra $A:=B\otimes_{\mathcal{O}H}\mathcal{O}G$ has a $G\times G$-stable $\mathcal{O}$-basis consisting of $\bar G$-homogeneous elements. Further, we assume that $A$ is projective regarded as a left and as a right  $\mathcal{O}G$-module. By these assumptions, $B$ is an $H$-interior permutation $G$-algebra, and it is projective both as a left and a right $\mathcal{O}H$-module.
\end{nim}



\begin{nim} We fix a normal subgroup $N$ of $G$ that contains $P$ and a point $\beta$  of $N$  on $B$ with defect group $P.$ Then our assumptions and \cite[Theorem 3.1 ]{COTO} imply that  $\tilde{\beta}:=\Br_P(\beta)$ is a point of $N_N(P)$ on $\tilde{B}(P)$ with defect group $P.$

We adopt here the definition of covering points from \cite{COCOVER}.
We say that the point $\alpha$ of $G$ on $A$ covers $\beta$ if $\alpha$ has defect group $P$ and for any $i\in \alpha$ there is an idempotent $j_1\in A^N$ that lies in the conjugacy class of $\beta$ and there is a primitive idempotent $f\in A^N$ belonging to a point with defect group $P$ such that $j_1f=fj_1=f$ and $if=fi=f.$

 Clearly  in this case we consider a particular setting in which a defect group of the points that are covered coincides with a defect group of the points that cover the given ones.
\end{nim}

Now we can state our result.

\begin{thm}  The Brauer homomorhism
\[\Br_P:A^P\to A(P)\] determines a bijective correspondence between the points of $A^G$ with defect group $P$ that cover $\beta$ and the points of $\tilde{A}(P)^{N_G(P)}$ with defect group $P$ that cover $\tilde{\beta}.$
\end{thm}

\begin{proof} Theorem \ref{theisomorphismofquotients} and \cite[Proposition 3.3]{PuZh} already provide a bijection between the points of $G$ on $A$ and the points of $N_G(P)$ on $\tilde{A}(P)$ with the same defect group $P.$ Even more, this bijection coincides with the bijection determined by the epimorphism given by the restriction of the Brauer homomorphims
\[\Br_P:A^{G}_P\to A(P)^{N_G(P)}_P. \]
Since $N$ is normal in $G$, hence $N_N(P)$ is also normal in $N_G(P)$, the fact that this bijection restricts to a bijection between the points that cover $\beta $ and $\tilde{\beta}$ is an easy verification given by \cite[Theorem 3.5]{COCOVER}.
\end{proof}


\bigskip

{\bf References}

\begin{thebibliography}{0}

\bibitem{COCOVER} T. Cocone\c t, {Covering points in permutation algebras}. Arch. Math, \textbf{100}  (2013), 107--113.
\bibitem{CoMaBasicGraded} T. Cocone\c t, A. Marcus, {Group graded basic Morita equivalences}.  arXiv:1607.02262v1 [math.RT].

\bibitem{COTO} T. Cocone\c t, C.-C. Todea, {The extended Brauer quotient of $N$-interior $G$-algebras}. J. Algebra \textbf{396} (2013), 10--17.
\bibitem{Ma} A. Marcus, { Representation Theory of Group Graded algebras}, Nova Science Publishers, 1999.
\bibitem{Pu} L. Puig, Blocks of finite groups: The Hyperfocal subalgebra of a Block, Springer Verlag Berlin Heidelberg, 2002.
\bibitem{PuZh} L. Puig, Y. Zhou, A local property of basic Morita equivalences, Math. Z. \textbf{256} (2007), 551--562.
\bibitem{PuZhII} L. Puig, Y. Zhou, A local property of basic Rickard equivalences, J. Algebra \textbf{322} (2009), 1946--1973.
\bibitem{The} J. Th\'{e}venaz, $G$-algebras and modular representation theory, Oxford Math. Monogr., Clarendon Press, Oxford 1995).
\end{thebibliography}
\end{document}